\documentclass[11pt]{amsart}

\usepackage{a4wide}
\usepackage{amssymb}
\usepackage{mathrsfs}
\usepackage{enumerate}
\usepackage{esint}
\usepackage{titletoc}
\usepackage[colorlinks=true, urlcolor=blue, linkcolor=blue, citecolor=black]{hyperref}
\usepackage[nameinlink]{cleveref}

\theoremstyle{plain}
\newtheorem{theorem}{Theorem}[section]
\newtheorem{lemma}[theorem]{Lemma}
\newtheorem{corollary}[theorem]{Corollary}

\theoremstyle{definition}

\numberwithin{equation}{section}

\DeclareMathOperator*{\esssup}{ess\,sup}
\DeclareMathOperator*{\essinf}{ess\,inf}
\DeclareMathOperator{\osc}{osc}

\makeatletter
\@namedef{subjclassname@2020}{\textup{2020} Mathematics Subject Classification}
\makeatother

\title[BBM theorem for fractional Sobolev spaces with variable exponents]{Bourgain, Brezis and Mironescu theorem for fractional Sobolev spaces with variable exponents}

\author{Minhyun Kim}
\address{Fakult\"at f\"ur Mathematik, Universit\"at Bielefeld, 33615 Bielefeld, Germany}
\email{minhyun.kim@uni-bielefeld.de}

\thanks{Minhyun Kim gratefully acknowledges financial support by the German Research Foundation (GRK 2235 - 282638148).}

\begin{document}

\begin{abstract}
A Bourgain--Brezis--Mironescu-type theorem for fractional Sobolev spaces with variable exponents is established for sufficiently regular functions. We prove, however, that a limiting embedding theorem for these spaces fails to hold in general.
\end{abstract}

\maketitle


\section{Introduction} \label{sec:introduction}


In a celebrated work \cite{BBM01}, Bourgain, Brezis and Mironescu study the asymptotic behavior of the fractional Sobolev seminorms when the order of differentiability approaches one. Their results are concerned with smooth bounded domains, but the same arguments work for $W^{1, p}$-extension domain. More precisely, if $\Omega \subset \mathbb{R}^{n}$ is a $W^{1, p}$-extension domain, then an equality
\begin{equation} \label{eq:BBM}
\lim_{s \nearrow 1} s(1-s) \int_{\Omega} \int_{\Omega} \frac{|u(x)-u(y)|^{p}}{|x-y|^{n+sp}} \,\mathrm{d}y \,\mathrm{d}x = K_{n, p} \int_{\Omega} |\nabla u(x)|^{p} \,\mathrm{d}x
\end{equation}
holds for any $u \in W^{1, p}(\Omega)$, where $n \in \mathbb{N}$, $p \in (1, +\infty)$ and
\begin{equation} \label{eq:BBM-const}
K_{n, p} = \frac{1}{p} \int_{\mathbb{S}^{n-1}} |\omega_{n}|^{p} \,\mathrm{d}\mathcal{H}^{n-1}(\omega) = 2\pi^{\frac{n-1}{2}} \frac{\Gamma(\frac{p+1}{2})}{\Gamma(\frac{n+p}{2})} \frac{1}{p}.
\end{equation}
The equality \eqref{eq:BBM} suggests a nonlocal approximation of the Sobolev norm and yields a limiting embedding $W^{s, p}(\Omega) \to W^{1, p}(\Omega)$ as $s \nearrow 1$. Note that $\Omega$ is allowed to be unbounded. For the second equality in \eqref{eq:BBM-const}, see, for instance, \cite[Lemma 2.1]{IN10} or \cite[Proposition 3.9]{FG20}.

In this paper, we consider the Sobolev and fractional Sobolev spaces with variable exponents. The Sobolev spaces with variable exponents $W^{1, p(\cdot)}$ appear naturally from variational problems \cite{AM01,Alk97,CPC97,FZ99,Zhi97}, modeling of the electrorheological fluids \cite{RR96,RR01,Ruz00}, thermorheological fluids \cite{AR06} and image processing \cite{AMS08,BCE+09,CLR06,LLP10}. We refer the reader to the books \cite{CUF13,DHHR11} for a comprehensive study. The fractional Sobolev spaces with variable exponents $W^{s, p(\cdot, \cdot)}$ are introduced recently in \cite{KRV17}, and studied extensively in various contexts: \cite{ABS19,ABS20,BB19,Bah18,BR18,BRW20,BT21,CK21,DPR17,HKL21,HKS19,OK21}.

This paper aims to answer a question whether a Bourgain--Brezis--Mironescu-type result is true for the fractional Sobolev spaces with variable exponents. On the one hand, the answer is affirmative in the sense that an equality similar to \eqref{eq:BBM}, which takes the variable exponents into account, holds for functions that are sufficiently regular. However, on the other hand, the answer is negative in the sense that this equality fails to holds in general for functions in Lebesgue spaces---even in Sobolev spaces---with variable exponents. It implies that a limiting embedding for the fractional Sobolev spaces with variable exponents, corresponding to $W^{s, p} \to W^{1, p}$, is not true in general.

Let $\Omega \subset \mathbb{R}^{n}$ be any open set (not necessarily bounded). Let $p: \Omega \times \Omega \to \mathbb{R}$ be a measurable function satisfying
\begin{equation} \label{eq:p-bound}
1 < p_{-} := \essinf_{x, y \in \Omega} p(x, y) \leq \esssup_{x, y \in \Omega} p(x, y) =: p_{+} < +\infty,
\end{equation}
and define $\bar{p}(x) := p(x, x)$. Let $L^{\bar{p}(\cdot)}(\Omega)$, $W^{1, \bar{p}(\cdot)}(\Omega)$ and $W^{s, p(\cdot, \cdot)}(\Omega)$ denote the Lebesgue, Sobolev and fractional Sobolev spaces with variable exponents, respectively. See \Cref{sec:preliminaries} for the precise definitions. We begin with an extension of \eqref{eq:BBM} to an equality which takes a variable exponent $p$ into account.

\begin{theorem} \label{thm:conv}
Assume that $p$ satisfies \eqref{eq:p-bound} and
\begin{equation} \label{eq:p-log-Holder}
r^{-\osc_{B_{r}(x)} p(x, \cdot)} \leq L \quad \text{for all } x \in \Omega \text{ and } 0 < r < \min \lbrace \mathrm{dist}(x, \partial \Omega), 1 \rbrace
\end{equation}
for some constant $L \geq 1$. Then for any $u \in C^{2}_{c}(\mathbb{R}^{n})$
\begin{equation} \label{eq:equality}
\lim_{s \nearrow 1} s(1-s) \int_{\Omega} \int_{\Omega} \frac{|u(x)-u(y)|^{p(x, y)}}{|x-y|^{n+sp(x, y)}} \,\mathrm{d}y \,\mathrm{d}x = \int_{\Omega} K_{n, \bar{p}(x)} |\nabla u(x)|^{\bar{p}(x)} \,\mathrm{d}x,
\end{equation}
where
\begin{equation*}
K_{n, \bar{p}(x)} = \frac{1}{\bar{p}(x)} \int_{\mathbb{S}^{n-1}} |\omega_{n}|^{p(x)} \,\mathrm{d}\mathcal{H}^{n-1}(\omega) = 2\pi^{\frac{n-1}{2}} \frac{\Gamma(\frac{\bar{p}(x)+1}{2})}{\Gamma(\frac{n+\bar{p}(x)}{2})} \frac{1}{\bar{p}(x)}.
\end{equation*}
\end{theorem}

The condition \eqref{eq:p-log-Holder} says that for a fixed point $x \in \Omega$ the function $p(x, \cdot)$ is log-H\"older continuous at $x$. We emphasize that it does not impose any regularity on $x$-variable.

\Cref{thm:conv} is not a full extension of \eqref{eq:BBM} because \eqref{eq:equality} holds for functions that are sufficiently regular. As a natural extension of Bourgain--Brezis--Mironescu theorem, one may expect that the equality \eqref{eq:equality} holds for all functions $u \in W^{1, \bar{p}(\cdot)}(\Omega)$. However, this is not true in general even when the variable exponent $p$ is smooth, which is in sharp contrast to the case when $p$ is constant.

\begin{theorem} \label{thm:failure}
There exists a smooth variable exponent $p(\cdot, \cdot)$ such that
\begin{equation*}
W^{1, \bar{p}(\cdot)}(\Omega) \not\subset W^{s, p(\cdot, \cdot)}(\Omega)
\end{equation*}
for all $s \in (0,1)$.
\end{theorem}

The following result is an immediate consequence of \Cref{thm:failure}, which shows that a limiting embedding $W^{s, p(\cdot, \cdot)} \to W^{1, \bar{p}(\cdot)}$ as $s \nearrow 1$ is not true in general.

\begin{corollary} \label{cor:failure}
There exists a smooth variable exponent $p$ such that the equality \eqref{eq:equality} does not hold for some $u \in W^{1, \bar{p}(\cdot)}(\Omega)$.
\end{corollary}

\Cref{cor:failure} tells us that the target space $W^{1, \bar{p}(\cdot)}$ is too large for \eqref{eq:equality} to be true in general. The failure of $W^{1, \bar{p}(\cdot)}$ can be explained by the following observation: let us forget about $p(\cdot, \cdot)$ and suppose that a variable exponent $\bar{p}(\cdot)$ is given. What is the natural fractional Sobolev space as an interpolation of $L^{\bar{p}(\cdot)}$ and $W^{1, \bar{p}(\cdot)}$? Since there are too many choices of $p(\cdot, \cdot)$ satisfying \eqref{eq:p-bound} and $p(x, x) = \bar{p}(x)$, there is no way to find the natural interpolation space. This is because the spaces $L^{\bar{p}(\cdot)}$ and $W^{1, \bar{p}(\cdot)}$ do not contain any information about $p(\cdot, \cdot)$.

The following result nevertheless shows that some Sobolev functions, in a smaller space than $W^{1, \bar{p}(\cdot)}$, satisfy the equality \eqref{eq:equality}.

\begin{corollary} \label{cor:conv}
Let $\Omega$ be a $W^{1, p_{+}}$- and $W^{1, p_{-}}$-extension domain. Assume that $p$ satisfies \eqref{eq:p-bound} and \eqref{eq:p-log-Holder}. Then the equality \eqref{eq:equality} holds for all $u \in W^{1, p_{+}}(\Omega) \cap W^{1, p_{-}}(\Omega)$.
\end{corollary}

It is worth mentioning the results in \cite{CFS21,FS22}, where the authors provide nonlocal approximations to the Sobolev norms with variable exponents $p(x)$. Their approximations involve a suitable class of nonlocal nonconvex energy functionals, generalizing the constant case \cite{BN20,Ngu06,Ngu08}, and their proofs make use of the Hardy--Littlewood maximal function with variable exponent $p(x)$. Since the maximal inequality with variable exponents in a modular form fails in general \cite{Ler05}, these approximations require an additional assumption $u \in W^{1, p_{+}} \cap W^{1, p_{-}}$ as well.

Let us also mention other related results in the literature. The Bourgain--Brezis--Mironescu theorem is extended to fractional Sobolev norms with L\'evy-type measures \cite{FG20,FGKV20}, Orlicz growth \cite{ACPS20a,FBS19} and anisotropic structure \cite{CDFB22}. The asymptotic behavior of the fractional Sobolev seminorms as $s \searrow 0$ is first established in \cite{MS02}, and extended those with Orlicz growth \cite{ACPS20b} and anisotropic structure \cite{CDFB22}.

The paper is organized as follows. In \Cref{sec:preliminaries}, we recall definitions of the Lebesgue, Sobolev and fractional Sobolev spaces with variable exponents. The asymptotic behavior of the fractional Sobolev seminorms with variable exponents (in a modular form), for regular functions, is proved in \Cref{sec:limit}. \Cref{sec:failure} is devoted to the proof of \Cref{thm:failure} and \Cref{cor:failure}.


\section{Preliminaries} \label{sec:preliminaries}


Let us briefly recall definitions of the Lebesgue, Sobolev and fractional Sobolev spaces with variable exponents. Let $\Omega \subset \mathbb{R}^{n}$ be an open set and let $q: \Omega \to \mathbb{R}$ be a measurable function satisfying
\begin{equation*}
1 < \essinf_{x \in \Omega} q(x) \leq \esssup_{x \in \Omega} q(x) < +\infty.
\end{equation*}
The space
\begin{equation*}
L^{q(\cdot)}(\Omega) = \left\lbrace u: \Omega \to \mathbb{R} \text{ measurable}: \varrho_{q(\cdot)}(u/\lambda) < +\infty \text{ for some } \lambda > 0 \right\rbrace
\end{equation*}
equipped with the norm 
\begin{equation*}
\|u\|_{q(\cdot)} := \|u\|_{L^{q(\cdot)}(\Omega)} := \inf \left\lbrace \lambda > 0: \varrho_{q(\cdot)}(u/\lambda) \leq 1 \right\rbrace
\end{equation*}
is called the {\it variable exponent Lebesgue space}, where
\begin{equation*}
\varrho_{q(\cdot)}(u) := \varrho_{L^{q(\cdot)}(\Omega)}(u) := \int_{\Omega} |u(x)|^{q(x)}\,\mathrm{d}x
\end{equation*}
is a modular. The {\it Sobolev space with variable exponent} is defined by
\begin{equation*}
W^{1, q(\cdot)}(\Omega) = \left\lbrace u \in L^{q(\cdot)}(\Omega): D_{i}u \in L^{q(\cdot)}(\Omega) \text{ for all } i = 1, \dots, n \right\rbrace
\end{equation*}
with the norm
\begin{equation*}
\|u\|_{1, q(\cdot)} := \|u\|_{W^{1, q(\cdot)}(\Omega)} := \|u\|_{q(\cdot)} + \sum_{1 \leq i \leq n} \|D_{i} u\|_{q(\cdot)}.
\end{equation*}
It is well known that $L^{q(\cdot)}(\Omega)$ and $W^{1, q(\cdot)}(\Omega)$ are Banach spaces, see \cite{CUF13,DHHR11,FZ01,KR91} for instance.

The fractional Sobolev spaces with variable exponents are first introduced in \cite{KRV17} and studied in various contexts. Let $p: \Omega \times \Omega \to \mathbb{R}$ be a measurable function satisfying \eqref{eq:p-bound} and define $\bar{p}(x) = p(x, x)$. For $s \in (0,1)$, the {\it fractional Sobolev space with variable exponent} is defined by
\begin{equation*}
W^{s, p(\cdot, \cdot)}(\Omega) = \left\lbrace u \in L^{\bar{p}(\cdot)}(\Omega): \varrho_{s, p(\cdot, \cdot)}(u/\lambda) < + \infty \text{ for some } \lambda > 0 \right\rbrace
\end{equation*}
with the norm
\begin{equation*}
\|u\|_{s, p(\cdot, \cdot)} := \|u\|_{W^{s, p(\cdot, \cdot)}(\Omega)} := \|u\|_{p(\cdot)} + [u]_{s, p(\cdot, \cdot)},
\end{equation*}
where
\begin{equation*}
[u]_{s, p(\cdot, \cdot)} := [u]_{W^{s,p(\cdot,\cdot)}(\Omega)} := \inf \left\lbrace \lambda > 0: \varrho_{s,p(\cdot,\cdot)}(u/\lambda) \leq 1 \right\rbrace
\end{equation*}
is a seminorm and
\begin{equation*}
\varrho_{s, p(\cdot, \cdot)}(u) := \varrho_{W^{s,p(\cdot,\cdot)}(\Omega)}(u) := s(1-s) \int_{\Omega} \int_{\Omega} \frac{|u(x)-u(y)|^{p(x,y)}}{|x-y|^{n+sp(x,y)}} \,\mathrm{d}y\,\mathrm{d}x.
\end{equation*}
It is easy to check that $W^{s, p(\cdot, \cdot)}(\Omega)$ is a Banach space by following the arguments in \cite[Proposition 4.24]{DD12}. Obviously, we have $W^{1, p_{+}}(\Omega) \cap W^{1, p_{-}}(\Omega) \subset W^{s, p_{+}}(\Omega) \cap W^{s, p_{-}}(\Omega) \subset W^{s, p(\cdot, \cdot)}(\Omega)$. Let us collect some useful properties for the the fractional Sobolev spaces with variables exponents. They follow immediately from the definition.

\begin{lemma} \label{lem:modular1}
Let $u \in W^{s, p(\cdot, \cdot)}(\Omega)$, then
\begin{equation*}
\min\left\lbrace [u]_{s, p(\cdot, \cdot)}^{p_{+}}, [u]_{s, p(\cdot, \cdot)}^{p_{-}} \right\rbrace \leq \varrho_{s, p(\cdot, \cdot)}(u) \leq \max\left\lbrace [u]_{s, p(\cdot, \cdot)}^{p_{+}}, [u]_{s, p(\cdot, \cdot)}^{p_{-}} \right\rbrace.
\end{equation*}
\end{lemma}

\begin{lemma} \label{lem:modular2}
Let $u, u_{k} \in W^{s, p(\cdot, \cdot)}(\Omega)$ for $k \in \mathbb{N}$. Then the following are equivalent:
\begin{enumerate}[(i)]
\item
$[u_{k}-u]_{s, p(\cdot, \cdot)} \to 0$ as $k \to \infty$,
\item
$\varrho_{s, p(\cdot, \cdot)}(u_{k}-u) \to 0$ as $k \to \infty$.
\end{enumerate}
In particular, if one of the assertions is satisfied, then $\varrho_{s, p(\cdot, \cdot)}(u_{k}) \to \varrho_{s, p(\cdot, \cdot)}(u)$ as $k \to \infty$.
\end{lemma}


\section{Limiting behavior of norms for regular functions} \label{sec:limit}


In this section, we study the limiting behavior of fractional Sobolev seminorms  with variable exponents (in a modular form) for sufficiently regular functions. Let us first prove \Cref{thm:conv}, which follows from the following two lemmas: \Cref{lem:liminf} and \Cref{lem:limsup}. Throughout the paper, we denote by $C$ strictly positive constants whose exact values are not important. These constants might change from line to line.

\begin{lemma} \label{lem:liminf}
Let $\Omega \subset \mathbb{R}^{n}$ be open and assume that $p$ satisfies \eqref{eq:p-bound} and \eqref{eq:p-log-Holder}. Then
\begin{equation} \label{eq:liminf}
\int_{\Omega} K_{n, \bar{p}(x)} |\nabla u(x)|^{\bar{p}(x)} \,\mathrm{d}x \leq \liminf_{s \nearrow 1} s(1-s) \int_{\Omega} \int_{\Omega} \frac{|u(x)-u(y)|^{p(x, y)}}{|x-y|^{n+sp(x, y)}} \,\mathrm{d}y \,\mathrm{d}x
\end{equation}
for any $u \in C^{2}(\overline{\Omega})$.
\end{lemma}

\begin{proof}
Let us fix $\varepsilon > 0$, $\theta_{0} \in (0,1)$, $R > 0$ and $x \in \Omega \cap B_{R}$. Since $u \in C^{2}(\overline{\Omega})$, there exists a constant $C > 0$, depending on $R$ and $\|u\|_{C^{2}(\overline{\Omega})}$, such that
\begin{equation} \label{eq:second-difference}
|u(x+h)-u(x)-\nabla u(x) \cdot h| \leq C |h|^{2}
\end{equation}
for all $|h| < \mathrm{dist}(x, \partial \Omega)$. Since $p(x, \cdot)$ is continuous at $x$, we find $r_{1} \in (0, \mathrm{dist}(x, \partial \Omega))$, depending on $\varepsilon$, $x$, $|\nabla u(x)|$ and $p$, such that
\begin{equation} \label{eq:r1}
(1-\varepsilon) |\nabla u(x)|^{p(x, x)} \leq |\nabla u(x)|^{p(x, y)} \leq (1+\varepsilon) |\nabla u(x)|^{p(x, x)}
\end{equation}
holds for all $y \in B_{r_{1}}(x)$. Moreover, by continuity of $p(x, \cdot)$ at $x$ again, there exists $r_{2} \in (0, \mathrm{dist}(x, \partial \Omega))$, depending on $\varepsilon$, $x$ and $\theta_0$, such that
\begin{equation} \label{eq:r2}
(1-\varepsilon) \theta^{p(x, x)} \leq \theta^{p(x, y)} \leq (1+ \varepsilon) \theta^{p(x, x)}
\end{equation}
holds for all $y \in B_{r_{2}}(x)$ and $\theta \in [\theta_{0}, 1]$. We set $r_{0} = \min\lbrace r_{1}, r_{2}, 1 \rbrace$ and let $r \in (0, r_{0})$.

For $x \in \Omega \cap B_{R}$ and $h \in B_{r}$, we have from \eqref{eq:second-difference}
\begin{equation*}
|\nabla u(x) \cdot h| \leq |u(x+h)-u(x)| + C |h|^2.
\end{equation*}
Using an inequality
\begin{equation} \label{eq:alg-ineq}
(a+b)^{q} \leq (1+\varepsilon)a^{q} + C(q, \varepsilon) b^{q}, \quad a, b \geq 0, q > 1,
\end{equation}
we obtain
\begin{equation*}
|\nabla u(x) \cdot h|^{p(x, x+h)} \leq (1+\varepsilon) |u(x+h)-u(x)|^{p(x, x+h)} + C |h|^{2p(x, x+h)}
\end{equation*}
for some $C = C(R, \|u\|_{C^{2}(\overline{\Omega})}, p_{+}, p_{-}, \varepsilon) > 0$. Multiplying by $|h|^{-n-sp(x, x+h)}$ and integrating over $B_{r}$ with respect to $h$, we have
\begin{equation} \label{eq:liminf-I}
\begin{split}
I
&:= \int_{B_{r}} \frac{|\nabla u(x) \cdot h|^{p(x, x+h)}}{|h|^{n+sp(x,x+h)}} \,\mathrm{d}h \\
&\leq (1+\varepsilon) \int_{B_{r}} \frac{|u(x+h)-u(x)|^{p(x, x+h)}}{|h|^{n+sp(x, x+h)}} \,\mathrm{d}h + C \int_{B_{r}} \frac{|h|^{2p(x, x+h)}}{|h|^{n+sp(x, x+h)}} \,\mathrm{d}h
\end{split}
\end{equation}
for each $x \in \Omega \cap B_{R}$.

Let us estimate $I$ from below. By using the spherical coordinates, we write
\begin{equation*}
I = \int_{0}^{r} \int_{\mathbb{S}^{n-1}} |\nabla u(x) \cdot \omega|^{p(x, x+\rho \omega)} \rho^{-1+(1-s)p(x, x+\rho \omega)} \,\mathrm{d}\mathcal{H}^{n-1}(\omega) \,\mathrm{d}\rho.
\end{equation*}
We use the change of variables $\omega = A\xi$, where $A \in O(n)$ is a rotation satisfying $\nabla u(x) \cdot \omega = |\nabla u(x)| \xi_{n}$, so that we have
\begin{equation} \label{eq:I}
I = \int_{0}^{r} \int_{\mathbb{S}^{n-1}} |\nabla u(x)|^{p(x, x+\rho A\xi)} |\xi_{n}|^{p(x, x+\rho A \xi)} \rho^{-1+(1-s)p(x, x+\rho A\xi)} \,\mathrm{d}\mathcal{H}^{n-1}(\xi) \,\mathrm{d}\rho.
\end{equation}
Using \eqref{eq:r1} and \eqref{eq:r2}, we obtain
\begin{equation*}
I \geq (1-\varepsilon)^2 \left( \int_{0}^{r} \int_{\lbrace |\xi_{n}| \geq \theta_{0} \rbrace} |\xi_{n}|^{\bar{p}(x)} \rho^{-1+(1-s)p(x, x+\rho A\xi)} \,\mathrm{d}\mathcal{H}^{n-1}(\xi) \,\mathrm{d}\rho \right) |\nabla u(x)|^{\bar{p}(x)}.
\end{equation*}
Moreover, since the assumption \eqref{eq:p-log-Holder} yields
\begin{equation} \label{eq:log-Holder}
L^{-1} \leq \min\lbrace 1, \rho^{\osc_{B_{\rho}(x)}p(x, \cdot)} \rbrace \leq \rho^{p(x, x+\rho A\xi)-\bar{p}(x)} \leq \max\lbrace 1, \rho^{-\osc_{B_{\rho}(x)}p(x, \cdot)}\rbrace \leq L,
\end{equation}
$I$ can be further estimated as
\begin{equation*}
\begin{split}
I
&\geq \frac{(1-\varepsilon)^2}{L^{1-s}} \left( \int_{0}^{r} \int_{\lbrace |\xi_{n}| \geq \theta_{0} \rbrace} |\xi_{n}|^{\bar{p}(x)} \rho^{-1+(1-s) \bar{p}(x)} \,\mathrm{d}\mathcal{H}^{n-1}(\xi) \,\mathrm{d}\rho \right) |\nabla u(x)|^{\bar{p}(x)} \\
&\geq \frac{(1-\varepsilon)^2}{(1-s)L^{1-s}} \left( \int_{\lbrace |\xi_{n}| \geq \theta_{0} \rbrace} |\xi_{n}|^{\bar{p}(x)} \,\mathrm{d}\mathcal{H}^{n-1}(\xi) \right) \frac{r^{(1-s)\bar{p}(x)}}{\bar{p}(x)} |\nabla u(x)|^{\bar{p}(x)}.
\end{split}
\end{equation*}
Since
\begin{equation*}
\begin{split}
\int_{\lbrace |\xi_{n}| \geq \theta_{0} \rbrace} |\xi_{n}|^{\bar{p}(x)} \,\mathrm{d}\mathcal{H}^{n-1}(\xi)
&\geq \int_{\mathbb{S}^{n-1}} |\xi_{n}|^{\bar{p}(x)} \,\mathrm{d}\mathcal{H}^{n-1}(\xi) - \int_{\lbrace |\xi_{n}| < \theta_{0} \rbrace} \theta_{0}^{p_{-}} \,\mathrm{d}\mathcal{H}^{n-1}(\xi) \\
&\geq \bar{p}(x) K_{n, \bar{p}(x)} - \theta_{0} |\mathbb{S}^{n-1}|,
\end{split}
\end{equation*}
we arrive at
\begin{equation} \label{eq:liminf-I-lower}
I \geq \frac{(1-\varepsilon)^2}{(1-s)L^{1-s}} \left( \bar{p}(x) K_{n, \bar{p}(x)} - \theta_0 |\mathbb{S}^{n-1}| \right) \frac{r^{(1-s)\bar{p}(x)}}{\bar{p}(x)} |\nabla u(x)|^{\bar{p}(x)}.
\end{equation}

For the last term on the right-hand side of \eqref{eq:liminf-I}, we use the assumption \eqref{eq:p-bound} to have
\begin{equation} \label{eq:liminf-I-upper}
\int_{B_{r}} |h|^{-n+(2-s)p(x, x+h)} \,\mathrm{d}h \leq \int_{B_{r}} |h|^{-n+(2-s)p_{-}} \,\mathrm{d}h \leq |\mathbb{S}^{n-1}| \frac{r^{(2-s)p_{-}}}{(2-s)p_{-}}.
\end{equation}
Therefore, combining \eqref{eq:liminf-I}, \eqref{eq:liminf-I-lower}, \eqref{eq:liminf-I-upper} and integrating over $\Omega \cap B_{R}$ yield
\begin{equation*}
\begin{split}
&(1-\varepsilon)^{2} \frac{s}{L^{1-s}} \int_{\Omega \cap B_{R}} \left( \bar{p}(x) K_{n, \bar{p}(x)} - \theta_{0} |\mathbb{S}^{n-1}| \right) \frac{r^{(1-s)\bar{p}(x)}}{\bar{p}(x)} |\nabla u(x)|^{\bar{p}(x)} \,\mathrm{d}x \\
&\leq (1+\varepsilon) s(1-s) \int_{\Omega} \int_{\Omega} \frac{|u(x)-u(y)|^{p(x, y)}}{|x-y|^{n+sp(x, y)}} \,\mathrm{d}y \,\mathrm{d}x + C \frac{s(1-s)}{2-s} r^{(2-s)p_{-}}
\end{split}
\end{equation*}
for some $C > 0$ independent of $s$. We take $\liminf_{s \nearrow 1}$ on both sides to deduce
\begin{equation*}
\begin{split}
&(1-\varepsilon)^{2} \int_{\Omega \cap B_{R}} \left( \bar{p}(x) K_{n, \bar{p}(x)} - \theta_{0} |\mathbb{S}^{n-1}| \right) \frac{1}{\bar{p}(x)} |\nabla u(x)|^{\bar{p}(x)} \,\mathrm{d}x \\
&\leq (1+\varepsilon) \liminf_{s \nearrow 1} s(1-s) \int_{\Omega} \int_{\Omega} \frac{|u(x)-u(y)|^{p(x, y)}}{|x-y|^{n+sp(x, y)}} \,\mathrm{d}y \,\mathrm{d}x.
\end{split}
\end{equation*}
Since $\varepsilon$, $\theta_0$ and $R$ are arbitrarily chosen, we conclude \eqref{eq:liminf}.
\end{proof}

\begin{lemma} \label{lem:limsup}
Let $\Omega \subset \mathbb{R}^{n}$ be open and assume that $p$ satisfies \eqref{eq:p-bound} and \eqref{eq:p-log-Holder}. Then
\begin{equation} \label{eq:limsup}
\limsup_{s \nearrow 1} s(1-s) \int_{\Omega} \int_{\Omega} \frac{|u(x)-u(y)|^{p(x, y)}}{|x-y|^{n+sp(x, y)}} \,\mathrm{d}y \,\mathrm{d}x \leq \int_{\Omega} K_{n, \bar{p}(x)} |\nabla u(x)|^{\bar{p}(x)} \,\mathrm{d}x
\end{equation}
for any $u \in C^{2}_{c}(\mathbb{R}^n)$.
\end{lemma}

To prove \Cref{lem:limsup}, we use the dominated convergence theorem. For this purpose, we first prove the following lemma.

\begin{lemma} \label{lem:DCT}
Assume that $p$ satisfies \eqref{eq:p-bound}. Let $s_{0} \in (0,1)$ and $u \in C^{1}_{c}(\mathbb{R}^{n})$. Then, there exists a function $F \in L^{1}(\Omega)$ such that
\begin{equation} \label{eq:Fs}
F_{s}(x):= s(1-s) \int_{\Omega} \frac{|u(x)-u(y)|^{p(x, y)}}{|x-y|^{n+sp(x, y)}} \,\mathrm{d}y \leq F(x), \quad x \in \Omega,
\end{equation}
for all $s \in [s_{0} ,1)$.
\end{lemma}

\begin{proof}
Let us write $F_{s}(x) = I_{1} + I_{2}$, where
\begin{equation*}
\begin{split}
I_{1} &= s(1-s) \int_{\Omega \cap B_{1}(x)} \frac{|u(x)-u(y)|^{p(x, y)}}{|x-y|^{n+sp(x, y)}} \,\mathrm{d}y \quad\text{and}\\
I_{2} &= s(1-s) \int_{\Omega \setminus B_{1}(x)} \frac{|u(x)-u(y)|^{p(x, y)}}{|x-y|^{n+sp(x, y)}} \,\mathrm{d}y.
\end{split}
\end{equation*}
Using
\begin{equation*}
|u(y)-u(x)| \leq \int_{0}^{1} |\nabla u(x+t(y-x)) \cdot (y-x)| \,\mathrm{d}t \leq \|\nabla u\|_{L^{\infty}(\mathbb{R}^{n})} |y-x|
\end{equation*}
and \eqref{eq:p-bound}, we have
\begin{equation} \label{eq:DCT-I1}
\begin{split}
I_{1}
&\leq (1-s) \int_{\Omega \cap B_{1}(x)} \|\nabla u\|_{L^{\infty}(\mathbb{R}^{n})}^{p(x, y)} |x-y|^{-n+(1-s)p(x, y)} \,\mathrm{d}y \\
&\leq (1-s) \left( \|\nabla u\|_{L^{\infty}(\mathbb{R}^{n})}^{p_{+}} +  \|\nabla u\|_{L^{\infty}(\mathbb{R}^{n})}^{p_{-}} \right) \int_{B_{1}} |h|^{-n+(1-s)p_{-}} \,\mathrm{d}h \\
&= \frac{|\mathbb{S}^{n-1}|}{p_{-}} \left( \|\nabla u\|_{L^{\infty}(\mathbb{R}^{n})}^{p_{+}} +  \|\nabla u\|_{L^{\infty}(\mathbb{R}^{n})}^{p_{-}} \right).
\end{split}
\end{equation}
Moreover, using \eqref{eq:p-bound} again, we obtain
\begin{equation} \label{eq:DCT-I2}
\begin{split}
I_{2}
&\leq 2^{p_+} s \int_{\Omega \setminus B_{1}(x)} \|u\|_{L^{\infty}(\mathbb{R}^{n})}^{p(x, x+h)} |x-y|^{-n-sp(x, y)} \,\mathrm{d}y \\
&\leq 2^{p_+} s \left( \|u\|_{L^{\infty}(\mathbb{R}^{n})}^{p_{+}} + \|u\|_{L^{\infty}(\mathbb{R}^{n})}^{p_{-}} \right) \int_{\mathbb{R}^{n} \setminus B_{1}} |h|^{-n-sp_{-}} \,\mathrm{d}h \\
&= 2^{p_{+}} \frac{|\mathbb{S}^{n-1}|}{p_{-}} \left( \|u\|_{L^{\infty}(\mathbb{R}^{n})}^{p_{+}} + \|u\|_{L^{\infty}(\mathbb{R}^{n})}^{p_{-}} \right).
\end{split}
\end{equation}
Thus, it follows from \eqref{eq:DCT-I1} and \eqref{eq:DCT-I2}
\begin{equation} \label{eq:Fs-bound}
F_{s} = I_{1} + I_{2} \leq C \quad\text{in } \Omega
\end{equation}
for all $s \in (0,1)$, where $C$ is a constant depending on $n$, $p_{+}$, $p_{-}$ and $\|u\|_{C^1(\mathbb{R}^{n})}$. This finishes the proof when $\Omega$ is bounded.

Let us next consider the case when $\Omega$ is bounded. We prove that $F_{s}$ is bounded by an integrable function for all $s \in [s_{0},1)$. To this end, let us consider a large ball $B_{R/2}$ containing $\mathrm{supp} \,u$. We may assume that $R > 2$. Then, for $x \in \Omega \setminus B_{R}$, we have
\begin{equation*}
F_{s}(x) = s(1-s) \int_{\Omega \cap B_{R/2}} \frac{|u(y)|^{p(x, y)}}{|x-y|^{n+sp(x, y)}} \,\mathrm{d}y.
\end{equation*}
Since $|x-y| \geq |x|-R/2 \geq |x|/2 \geq 1$, using \eqref{eq:p-bound} we obtain
\begin{equation} \label{eq:Fs-decay}
F_{s}(x) \leq |B_{R/2}| \left( \|u\|_{L^{\infty}(\mathbb{R}^{n})}^{p_{+}} + \|u\|_{L^{\infty}(\mathbb{R}^{n})}^{p_{-}} \right) \left( \frac{2}{|x|} \right)^{n+s_{0} p_{-}} \quad\text{for } x \in \Omega \setminus B_{R}.
\end{equation}
Combining \eqref{eq:Fs-bound} and \eqref{eq:Fs-decay}, we conclude that
\begin{equation*}
F_{s}(x) \leq C \min \left\lbrace \frac{1}{R^{n+s_{0}p_{-}}}, \frac{1}{|x|^{n+s_{0}p_{-}}} \right\rbrace =: F(x) \in L^{1}(\Omega),
\end{equation*}
where $C$ is a constant depending on $n$, $p_{+}$, $p_{-}$, $s_{0}$ and $u$.
\end{proof}

In the next lemma, we find a pointwise limit of $F_{s}$ as $s \nearrow 1$.

\begin{lemma} \label{lem:limit-Fs}
Assume that $p$ satisfies \eqref{eq:p-bound} and \eqref{eq:p-log-Holder}. Let $u \in C^{2}_{c}(\mathbb{R}^{n})$ and let $F_{s}$ be given as in \eqref{eq:Fs}. Then
\begin{equation*}
\lim_{s \nearrow 1} F_{s}(x) = K_{n, \bar{p}(x)} |\nabla u(x)|^{\bar{p}(x)}
\end{equation*}
for all $x \in \Omega$.
\end{lemma}

\begin{proof}
Let $u \in C^{2}_{c}(\mathbb{R}^{n})$, then there exists a constant $C > 0$, depending on $\|u\|_{C^{2}(\mathbb{R}^{n})}$, such that
\begin{equation} \label{eq:gradient}
|u(y)-u(x)| \leq |\nabla u(x) \cdot (y-x)| + C|y-x|^2
\end{equation}
holds for all $x, y \in \mathbb{R}^{n}$. Let $\varepsilon > 0$ and $\theta_{0} \in (0,1)$. Similarly as in the proof of \Cref{lem:liminf}, for a fixed point $x \in \Omega$ we find $r_{0} \in (0,1)$ such that \eqref{eq:r1} and \eqref{eq:r2} hold for all $y \in B_{r_{0}}(x)$. Then, using \eqref{eq:gradient} and \eqref{eq:alg-ineq} we have
\begin{equation} \label{eq:limsup-I12}
\begin{split}
\int_{\Omega} \frac{|u(x)-u(y)|^{p(x, y)}}{|x-y|^{n+sp(x, y)}} \,\mathrm{d}y
&\leq (1+\varepsilon) \int_{B_{r}(x)} \frac{|\nabla u(x) \cdot (y-x)|^{p(x, y)}}{|x-y|^{n+sp(x, y)}} \,\mathrm{d}y \\
&\quad + C \int_{B_{r}(x)} \frac{|x-y|^{2p(x, y)}}{|x-y|^{n+sp(x, y)}} \,\mathrm{d}y \\
&\quad + \int_{\Omega \setminus B_{r}(x)} \frac{|u(x)-u(y)|^{p(x, y)}}{|x-y|^{n+sp(x, y)}} \,\mathrm{d}y \\
&=: I_{1} + I_{2} + I_{3}
\end{split}
\end{equation}
for $r \in (0, r_{0})$.

Let us estimate $I_{1}$, $I_{2}$ and $I_{3}$. We first recall from \eqref{eq:I} and \eqref{eq:liminf-I-upper} that
\begin{equation*}
I_{1} = (1+\varepsilon) \int_{0}^{r} \int_{\mathbb{S}^{n-1}} |\nabla u(x)|^{p(x, x+\rho A\xi)} |\xi_{n}|^{p(x, x+\rho A \xi)} \rho^{-1+(1-s)p(x, x+\rho A\xi)} \,\mathrm{d}\mathcal{H}^{n-1}(\xi) \,\mathrm{d}\rho
\end{equation*}
and
\begin{equation} \label{eq:limsup-I2}
I_{2} \leq C \frac{r^{(2-s)p_-}}{2-s}
\end{equation}
hold. For $I_{1}$, we use \eqref{eq:r1} and \eqref{eq:log-Holder} to have
\begin{equation*}
I_{1} \leq (1+\varepsilon)^{2} L^{1-s} \left( \int_{0}^{r} \int_{\mathbb{S}^{n-1}} |\xi_{n}|^{p(x, x+\rho A\xi)} \,\mathrm{d}\mathcal{H}^{n-1}(\xi) \,\rho^{-1+(1-s)\bar{p}(x)} \,\mathrm{d}\rho \right) |\nabla u(x)|^{\bar{p}(x)}.
\end{equation*}
Moreover, by using \eqref{eq:r2} we obtain
\begin{equation*}
\begin{split}
\int_{\mathbb{S}^{n-1}} |\xi_{n}|^{p(x, x+\rho A\xi)} \,\mathrm{d}\mathcal{H}^{n-1}(\xi)
&\leq \int_{\lbrace |\xi_{n}| \geq \theta_{0} \rbrace} |\xi_n|^{p(x, x+\rho A\xi)} \,\mathrm{d}\mathcal{H}^{n-1}(\xi) + \theta_{0} |\mathbb{S}^{n-1}| \\
&\leq (1+\varepsilon) \bar{p}(x) K_{n, \bar{p}(x)} + \theta_{0} |\mathbb{S}^{n-1}|.
\end{split}
\end{equation*}
Thus, we deduce
\begin{equation} \label{eq:limsup-I1}
I_{1} \leq (1+\varepsilon)^{2} \frac{L^{1-s}}{1-s} \left( (1+\varepsilon) \bar{p}(x) K_{n, \bar{p}(x)} + \theta_{0} |\mathbb{S}^{n-1}| \right) \frac{r^{(1-s)\bar{p}(x)}}{\bar{p}(x)} |\nabla u(x)|^{\bar{p}(x)}.
\end{equation}
For $I_{3}$, it follows from $|x-y|^{-sp(x, y)} \leq r^{sp_{-} - sp(x, y)} |x-y|^{-sp_{-}} \leq r^{s(p_{-} - p_{+})} |x-y|^{-sp_{-}}$
\begin{equation} \label{eq:limsup-I3}
I_{3} \leq \frac{2^{p_{+}} |\mathbb{S}^{n-1}|}{sp_{-}} \left( \|u\|_{L^{\infty}(\mathbb{R}^{n})}^{p_{+}} + \|u\|_{L^{\infty}(\mathbb{R}^{n})}^{p_{-}} \right) r^{-sp_{+}}.
\end{equation}
Combining \eqref{eq:limsup-I12}--\eqref{eq:limsup-I3}, we arrive at
\begin{equation*}
\begin{split}
F_{s}(x)
&= s(1-s) (I_{1} + I_{2} + I_{3}) \\
&\leq (1+\varepsilon)^{2} sL^{1-s} \left( (1+\varepsilon) \bar{p}(x) K_{n, \bar{p}(x)} + \theta_{0} |\mathbb{S}^{n-1}| \right) \frac{r^{(1-s)\bar{p}(x)}}{\bar{p}(x)} |\nabla u(x)|^{\bar{p}(x)} \\
&\quad + C \frac{s(1-s)}{2-s} r^{(2-s)p_{-}} + \frac{2^{p_{+}} |\mathbb{S}^{n-1}|}{p_{-}} (1-s) \left( \|u\|_{L^{\infty}(\mathbb{R}^{n})}^{p_{+}} + \|u\|_{L^{\infty}(\mathbb{R}^{n})}^{p_{-}} \right) r^{-sp_{+}},
\end{split}
\end{equation*}
and hence
\begin{equation*}
\limsup_{s \nearrow 1} F_{s}(x) \leq (1+\varepsilon)^{2} \left( (1+\varepsilon) \bar{p}(x) K_{n, \bar{p}(x)} + \theta_{0} |\mathbb{S}^{n-1}| \right) \frac{1}{\bar{p}(x)} |\nabla u(x)|^{\bar{p}(x)}.
\end{equation*}
Since $\varepsilon$ and $\theta_{0}$ are arbitrarily chosen, we deduce $\limsup_{s \nearrow 1} F_{s}(x) \leq K_{n, \bar{p}(x)} |\nabla u(x)|^{\bar{p}(x)}$. Recall from \eqref{eq:liminf-I-lower}
\begin{equation*}
F_{s}(x) \geq (1+\varepsilon)(1-\varepsilon)^{2} \frac{s}{L^{1-s}} \left( \bar{p}(x) K_{n, \bar{p}(x)} - \theta_{0} |\mathbb{S}^{n-1}| \right) \frac{r^{(1-s)\bar{p}(x)}}{\bar{p}(x)} |\nabla u(x)|^{\bar{p}(x)},
\end{equation*}
which results in $\liminf_{s \nearrow 1} F_{s}(x) \geq K_{n, \bar{p}(x)} |\nabla u(x)|^{\bar{p}(x)}$, finishing the proof.
\end{proof}

\Cref{lem:limsup} follows from \Cref{lem:DCT} and \Cref{lem:limit-Fs}. Moreover, \Cref{thm:conv} is an immediate consequence of \Cref{lem:liminf}, \Cref{lem:limsup} and the dominated convergence theorem. Let us finish this section by providing the proof of \Cref{cor:conv}.

\begin{proof} [Proof of \Cref{cor:conv}]
Since $\Omega$ is a $W^{1, p_{+}}$- and $W^{1, p_{-}}$-extension domain, we may assume that $u \in W^{1, p_{+}}(\mathbb{R}^{n}) \cap W^{1, p_{-}}(\mathbb{R}^{n})$. There exists a sequence $\lbrace u_{k} \rbrace \subset C^{2}_{c}(\mathbb{R}^{n})$ such that $u_{k} \to u$ in $W^{1, p_{+}}(\mathbb{R}^{n})$ and $W^{1, p_{-}}(\mathbb{R}^{n})$ as $k \to \infty$. By the triangle inequality, we have
\begin{equation*}
\begin{split}
\left| \varrho_{W^{s, p(\cdot, \cdot)}(\Omega)}(u) - \tilde{\varrho}_{L^{\bar{p}(\cdot)}(\Omega)}(|\nabla u|) \right|
&\leq \left| \varrho_{W^{s, p(\cdot, \cdot)}(\Omega)}(u) - \varrho_{W^{s, p(\cdot, \cdot)}(\Omega)}(u_{k}) \right| \\
&\quad + \left| \varrho_{W^{s, p(\cdot, \cdot)}(\Omega)}(u_{k}) - \tilde{\varrho}_{L^{\bar{p}(\cdot)}(\Omega)}(|\nabla u_{k}|) \right| \\
&\quad + \left| \tilde{\varrho}_{L^{\bar{p}(\cdot)}(\Omega)}(|\nabla u_{k}|) - \tilde{\varrho}_{L^{\bar{p}(\cdot)}(\Omega)}(|\nabla u|) \right| \\
&=: I_{1} + I_{2} + I_{3},
\end{split}
\end{equation*}
where
\begin{equation*}
\tilde{\varrho}_{L^{\bar{p}(\cdot)}(\Omega)}(v) = \int_{\Omega} K_{n, \bar{p}}(x) |v(x)|^{\bar{p}(x)} \,\mathrm{d}x.
\end{equation*}
Since $W^{1, p_{+}}(\mathbb{R}^{n}) \cap W^{1, p_{-}}(\mathbb{R}^{n}) \subset W^{s, p(\cdot, \cdot)}(\mathbb{R}^{n}) \subset W^{s, p(\cdot, \cdot)}(\Omega)$, $u_{k} \to u$ in $W^{s, p(\cdot, \cdot)}(\Omega)$ as $k \to \infty$. Thus, by \Cref{lem:modular2}, we have $I_{1} \to 0$ as $k \to \infty$. Obviously, we have $I_{3} \to 0$ as $k \to \infty$. Therefore, for $\varepsilon > 0$, we obtain
\begin{equation*}
\left| \varrho_{W^{s, p(\cdot, \cdot)}(\Omega)}(u) - \tilde{\varrho}_{L^{\bar{p}(\cdot)}(\Omega)}(|\nabla u|) \right| \leq \varepsilon + I_{2}
\end{equation*}
for sufficiently large $k$. By \Cref{thm:conv} we arrive at
\begin{equation*}
\limsup_{s \nearrow 1} \left| \varrho_{W^{s, p(\cdot, \cdot)}(\Omega)}(u) - \tilde{\varrho}_{L^{\bar{p}(\cdot)}(\Omega)}(|\nabla u|) \right| \leq \varepsilon,
\end{equation*}
which conclude the theorem.
\end{proof}


\section{Failure of the limiting embedding} \label{sec:failure}


Having \Cref{thm:conv} at hand, one may expect the limiting embedding $W^{s, p(\cdot, \cdot)} \to W^{1, \bar{p}(\cdot)}$ as a natural extension of the limiting embedding $W^{s, p} \to W^{1, p}$. However, in this section, we provide an example of a variable exponent $p(\cdot, \cdot)$ such that the limiting embedding $W^{s, p(\cdot, \cdot)} \to W^{1, \bar{p}(\cdot)}$ fails. As explained in the introduction, the idea is to construct an unnatural fractional Sobolev space as one of the interpolations of the given Lebesgue and Sobolev spaces. Such a strange space can be constructed even from $L^{p}$ and $W^{1, p}$ with a constant $p$. In the following proof, we find a variable exponent $p(\cdot, \cdot)$ such that $\bar{p}(x) = p(x, x) \equiv p$ is a constant and that $W^{1, p} \not\subset W^{s, p(\cdot, \cdot)}$ for all $s \in (0,1)$.

\begin{proof} [Proof of \Cref{thm:failure}]
We may assume that $\Omega$ contains $\overline{B}_{2}$. The idea of proof is to construct a variable exponent $p$ so that $W^{s, p(\cdot, \cdot)}(\Omega)$ requires more integrability than $W^{1, \bar{p}(\cdot)}(\Omega)$. To this end, let $p: \Omega \times \Omega \to \mathbb{R}$ be a smooth function satisfying \eqref{eq:p-bound}. Slightly abusing notation, we assume that $p(x, y) = p(|x-y|)$ for all $x, y \in \Omega$. Then, $\bar{p} = \bar{p}(x) = p(x, x) = p(0)$ is a constant. For $q \in (1, n)$, we assume
\begin{equation} \label{eq:p}
1 < \bar{p} < n/q \quad\text{and}\quad p(r) \geq n/(q-1) \quad\text{for } r \geq 1.
\end{equation}
The assumption \eqref{eq:p} shows how $W^{s, p(\cdot, \cdot)}(\Omega)$ requires more integrability than $W^{1, \bar{p}}(\Omega)$.

We consider a smooth decreasing function $\phi: (0, +\infty) \to [0, +\infty)$ satisfying $\phi(t) = |t|^{1-q}$ for $t \in (0, 1]$, $0 \leq \phi(t) \leq 1$ for $t \in [1, 2]$ and $\phi(t) = 0$ for $t \in [2, +\infty)$. We claim that a function $u: \Omega \to [0, +\infty]$ defined by $u(x) = \phi(|x|)$ belongs to $W^{1, \bar{p}}(\Omega) \setminus W^{s, p(\cdot, \cdot)}(\Omega)$ for all $s \in (0,1)$. Indeed, a straightforward computation shows
\begin{equation*}
\int_{\Omega} |u(x)|^{\bar{p}} \,\mathrm{d}x \leq \int_{B_{1}} |x|^{-(q-1)\bar{p}} \,\mathrm{d}x + |B_{2} \setminus B_{1}| < +\infty
\end{equation*}
and
\begin{equation*}
\int_{\Omega} |\nabla u(x)|^{\bar{p}} \,\mathrm{d}x \leq (q-1)^{\bar{p}} \int_{B_{1}} |x|^{-q\bar{p}} \,\mathrm{d}x + C |B_{2} \setminus B_{1}| < + \infty.
\end{equation*}
Thus, $u \in W^{1, \bar{p}}(\Omega)$.

It remains to prove $u \not\in W^{s, p(\cdot, \cdot)}(\Omega)$ for all $s \in (0,1)$. Since $u=0$ outside $B_{2}$, we have
\begin{equation*}
\varrho_{s, p(\cdot, \cdot)}(u) \geq s(1-s) \int_{\Omega \setminus B_{2}} \int_{B_{1}} \frac{|y|^{-(q-1)p(|x-y|)}}{|x-y|^{n+sp(|x-y|)}} \,\mathrm{d}y \,\mathrm{d}x.
\end{equation*}
If $x \in \Omega \setminus B_{2}$ and $y \in B_{1}$, then $1 \leq |x-y| \leq |x|+1 \leq \frac{3}{2} |x|$. Thus, by using \eqref{eq:p-bound} and \eqref{eq:p}, we obtain
\begin{equation*}
\varrho_{s, p(\cdot, \cdot)}(u) \geq s(1-s) \int_{\Omega \setminus B_{2}} \int_{B_{1}} \frac{|y|^{-n}}{(\frac{3}{2}|x|)^{n+sp_+}} \,\mathrm{d}y \,\mathrm{d}x = +\infty.
\end{equation*}
Therefore, it follows from \Cref{lem:modular1} $u \notin W^{s, p(\cdot, \cdot)}(\Omega)$ for all $s \in (0,1)$.
\end{proof}



\end{document}